\documentclass[11pt, a4paper, oneside]{amsart} 

\usepackage{amsmath, amssymb, amsthm, amsrefs}
\usepackage{bbm}
\usepackage{mathrsfs}
\usepackage{graphicx}
\usepackage{xcolor}
\usepackage{layout}
\usepackage{bm}
\usepackage[italicdiff]{physics}
\usepackage{mathtools}
\usepackage{enumitem}

\usepackage{hyperref}
\hypersetup{
	colorlinks=true, 
	linktoc= page,   
	citecolor=blue
}

\allowdisplaybreaks[4]
\numberwithin{equation}{section}

\def\le{\leqslant}
\def\ge{\geqslant}

\def\R{\mathbb{R}}
\def\C{\mathbb{C}}

\def\Z{\mathbb{Z}}

\newcommand{\relmiddle}[1]{\mathrel{}\middle#1\mathrel{}}

\theoremstyle{plain}
\newtheorem{theorem}{Theorem}[section]
\newtheorem{definition}[theorem]{Definition}
\newtheorem{lemma}[theorem]{Lemma}

\theoremstyle{remark}
\newtheorem{remark}[theorem]{Remark}

\begin{document}
\title[]
{Threshold for the existence of scattering states for nonlinear Schr\"odinger equations without gauge invariance}
\author[H. Miyazaki]{Hayato Miyazaki}
\address[]{Teacher Training Courses, Faculty of Education, Kagawa University, Takamatsu, Kagawa 760-8522, Japan}
\email{miyazaki.hayato@kagawa-u.ac.jp}
\author[M. Sobajima]{Motohiro Sobajima}
\address[]{Department of Mathematics, Faculty of Science and Technology, Tokyo University of Science, 2641
Yamazaki, Noda-shi, Chiba, 278-8510, Japan}
\email{msobajima1984@gmail.com}
\date{}

\maketitle

\begin{abstract}
This paper is concerned with a threshold phenomenon for the existence of scattering states for nonlinear Schr\"odinger equations.
The nonlinearity includes a non-oscillatory term
of the order lower than the Strauss exponent.
We show that no scattering states exist for the equation in a weighted Sobolev space.
It is emphasized that our method admits initial data with good properties, such as compactly supported smooth functions.
The result indicates that the Strauss exponent acts as a threshold for the power of the nonlinearity that determines whether solutions scatter or not in the weighted space.
\end{abstract}

\section{Introduction}
In this paper we consider the nonlinear Schr\"odinger equation
\begin{align}
	i\partial_t u + \frac{1}{2} \Delta u = F(u), \quad (t,x) \in \R \times \R^{d}, \tag{NLS} \label{nls}
\end{align}
where $u = u(t,x)$ is a complex-valued unknown function.
The nonlinearity $F$ satisfies the following condition:
\begin{enumerate}[label=(A\arabic*), ref=A\arabic*]
\setlength{\itemsep}{2mm}
\item $F(0)=0$ \label{A1}
and there exists a constant $C>0$ such that
\[
	\abs{F(z) - F(w)} \le C (|z| + |w|)^{p-1} |z-w|
\]
for any $z$, $w \in \mathbb{C}$ and some $p>1$.
\item There exist $c_{0}>0$ and $\theta \in \R$ such that $\Re \left[ e^{i\theta} F(z) \right] \ge c_{0} |z|^{p}$ for all $z \in \C$. \label{f:con}
\end{enumerate}
As for the power of the nonlinearity, we further assume
\begin{align*}
	1 < p < p_{\mathrm{st}}(d) \coloneqq \frac{d+2 + \sqrt{d^{2}+ 12d+4}}{2d}.
\end{align*}
The exponent $p_{\mathrm{st}}(d)$ is called the Strauss exponent, which is defined by the positive root of the quadratic equation $d x^{2} - (d+2)x - 2 = 0$.

A typical example of the nonlinearity in our mind is the power-type nonlinearity $F(u) = \eta |u|^{p}$ with $\eta \in \C \setminus \{0\}$.
In this case, the condition \eqref{f:con} is satisfied with $c_{0} = |\eta|$ and $\theta = - \arg \eta$.
For more general nonlinearity, we can treat the homogenous nonlinearity $F$ of degree $p$; that is, $F$ satisfies
\begin{align}
	F(\lambda z) = \lambda^{p} F(z)
	\label{eq:cond1}
\end{align}
for any $z \in \C$ and $\lambda >0$.
Remark that, if $F$ satisfies \eqref{eq:cond1}, then the condition \eqref{A1} is equivalent to the fact that the $2 \pi$-periodic function
$\theta \mapsto F(e^{i \theta})$ is Lipschitz continuous (see \cite[Appendix A]{MM}).
Using the method of \cite{MM, MMU} via the Fourier series expansion, we can decompose the nonlinearity into
\begin{align}\label{eq:decomp}
	F(u) = g_0 |u|^{p} + g_1 |u|^{p-1}u + \sum_{n\neq 0,1} g_n |u|^{p-n}u^n, \quad
	g_n \coloneqq \frac1{2\pi} \int_0^{2\pi} F(e^{i\theta}) e^{-in\theta} d\theta.
\end{align}
The nonlinearity
\eqref{eq:decomp} satisfies \eqref{f:con} if $\{g_{n}\} \in \ell^{1}(\Z)$ and $|g_{0}| > \sum_{n\neq0} |g_n|$.

In the context of nonlinear Schr\"odinger equations,
the Strauss exponent $p_{\mathrm{st}}(d)$ was firstly identified by Strauss \cite{S81a}, as the lower bound on the power of nonlinearity for which
a scattering operator exists for small data in certain norms
(see also \cite{DL81, C03}).
In the case of \eqref{nls} with the gauge invariant nonlinearity
\begin{align}
	i \partial_t u + \frac{1}{2} \Delta u = \eta |u|^{p-1}u, \quad \eta \in \R \setminus \{0\}, \quad p>1,
	\label{ginls}
\end{align}
the asymptotic completeness of the scattering operator was proven by Tsutsumi \cite{YT85} under the condition $p>p_{\mathrm{st}}(d)$ (see also \cite{HT87}).
The critical case $p=p_{\mathrm{st}}(d)$ in \eqref{ginls} was treated by \cite{CW92, NO02}.
Under the sub-Strauss case $1+2/d < p < p_{\mathrm{st}}(d)$, in \cite{CW92, GOV94, NO02}, it was proven that \eqref{ginls} admits a scattering solution.
Remark that Strauss \cite{S74} discussed the non-existence of scattering states for the equation \eqref{ginls}
under the case $p \le 1+2/d$ (cf. \cite{B84, C03}).
A generalization of scattering to solitary waves or more general wave profiles was presented in \cite{MN21}.
The exponent $1+2/d$ is known as the critical exponent in view of the long-time behavior of solutions.
In fact, Tsutsumi and Yajima \cite{TY84} established the existence of scattering solutions to \eqref{ginls} for $p > 1 + 2/d$.
Further improvements of their result were investigated in \cite{BGTV23}.
Regarding \eqref{nls} with \eqref{A1}, $p>p_{\mathrm{st}}(d)$,
Kato \cite{Ka94} established the existence of the scattering operator for data in its domain, where a certain space-time norm of the corresponding asymptotically free solution is sufficiently small.
It should be mentioned that the method in \cite{Ka94}
does not depend on such a structure of the nonlinearity.
Hence this implies that the structure does not affect the long-time behavior of small solutions under the super-Strauss case $p>p_{\mathrm{st}}(d)$.
However, we have observed that
the Strauss exponent in the analysis of \eqref{nls}
has been regarded merely as the lower bound of the validity of concrete techniques, but it has not been recognized as the upper bound marking a qualitative change in the behavior of solutions.
This paper aims to reveal such a phenomenon in \eqref{nls}.

As is well-known, the Strauss exponent also appears in the analysis of
the initial-value problem of the semilinear wave equation
\begin{align}\label{NW}
	\tag{NW}
	\partial_t^2u-\Delta u=|u|^p, \quad(0,\infty)\times \R^d.
\end{align}
Roughly speaking, the number $p_{\rm st}(d-1)$ (together with $p_{\rm st}(0)=\infty$)
is the threshold for dividing the situation
of existence or non-existence of global-in-time solutions to \eqref{NW}
with (non-trivial) nonnegative initial data.
More precisely, non-trivial global-in-time solutions (with small amplitude) exist when $p>p_{\rm st}(d-1)$
and solutions of \eqref{NW} with nonnegative small initial data blow up in finite time
when $p \le p_{\rm st}(d-1)$.
For this direction, we refer to, for instance,
John \cite{John1979},
Kato \cite{Kato80},
Glassey \cite{Glassey81a, Glassey81b},
Sideris \cite{Sideris84},
Schaeffer \cite{Schaeffer85},
Rammaha \cite{Rammaha89},
Georgiev, Lindblad and Sogge \cite{GLS97},
Yordanov and Zhang \cite{YZ06}, and Zhou \cite{Zhou2007}.

Here we focus our attention to
the existence or non-existence of
\textit{scattering solutions} to \eqref{nls}.
It is not difficult to construct small data scattering solutions
to \eqref{nls}
in the super-Strauss case $p>p_{\rm st}(d)$ by the contraction mapping principle
with the application of (non-admissible) Strichartz' estimates, established in the aforementioned paper \cite{Ka94}.
Therefore the main consideration in
the present paper
is the non-existence of scattering solutions
in the sub-Strauss case $p<p_{\rm st}(d)$.
In this paper, in view of the strategy for the analysis of \eqref{NW},
we aim to show the non-existence of scattering states for \eqref{nls} below the Strauss exponent.
In contrast to the gauge-invariant case \eqref{ginls}, it is not expected that \eqref{nls} admits scattering solutions when $p< p_{\mathrm{st}}(d)$.
In fact, Shimomura \cite{S05} proved that the equation
\begin{align}
	i \partial_t u + \frac{1}{2} \Delta u = \rho |u|^p, \quad \rho \in \C \setminus \{0\}, \quad p>1, \label{ngnls}
\end{align}
with if $d=2$ and $p = 1+2/d = 2$,
has no non-trivial solution that behaves like a free solution (cf. \cite{ST06}).
Later on, Masaki and the first author \cite{MM2} extended their result. Namely,
regarding \eqref{nls} with the nonlinearity of the form \eqref{eq:decomp} with $g_0 \neq 0$, $g_1 \in \R$ and $\{g_{n}\} \in \ell^{1}(\Z)$,
they showed the non-existence of the solution that behaves like a free solution with or without a phase modification in any dimensions.

Also, with regard to \eqref{nls} with \eqref{eq:decomp},
global-in-time existence is not always true if $p< 1+4/d$ and $g_{0} \neq 0$.
In fact, Ikeda and Wakasugi \cite{IW13} found blow-up solutions to \eqref{ngnls}
when $p \le 1+2/d$.
In \cite{II15a}, Ikeda and Inui proved the existence of blow-up solutions to \eqref{ngnls} for $p < 1+4/d$ under slowly decaying initial data.
The results in \cite{IW13, II15a} were partially generalized by Fujiwara and Ozawa \cite{FO16} from the viewpoint of an ordinary differential equation for a spatial integral involving both $u$ and nonnegative rapidly decreasing functions.
In this paper, instead of considering the slowly decaying initial data in \cite{II15a}, we adopt a $L^{2}$-based weighted space as the class of initial data that naturally arises in the study of the scattering problem for nonlinear Schr\"odinger equations (see \cite{GOV94, NO02}).
Within this framework, we establish the non-existence of scattering states for \eqref{nls} under the condition \eqref{A1} and \eqref{f:con} with $p<p_{\mathrm{st}}(d)$.
This is in contrast to the case of gauge invariance \eqref{ginls}, which admits scattering
solutions in the weighted space.

\subsection{Main results}
To state the main result, let us introduce the notations used throughout this paper.
For any $p \ge 1$, $L^{p}(\Omega)$ denotes the usual Lebesgue space on $\Omega \subset \R^{d}$.
If $\Omega = \R$, we abbreviate $L^{p}(\R^{d})$ as $L^{p}$.
Set $\ev{a}=(1+|a|^2)^{1/2}$ for $a \in \R$.
$\mathcal{F}[u] = \widehat{u}$ is the usual Fourier transform of a function $u$ on $\R^d$.
For $s \in \R$,
the standard Sobolev space on $\R^{d}$ is defined by $H^{s} = H^{s}(\R^{d}) \coloneqq \{ u \in L^{2}(\mathbb{R}^d) \mid \norm{u}_{H^{s}} \coloneqq \norm{\ev{\nabla}^{s} u}_{L^2} < \infty \} $.
We denote the weighted Sobolev space by $\mathcal{F}H^{s} = \mathcal{F}H^{s}(\R^{d}) \coloneqq \{ u \in L^{2}(\mathbb{R}^d) \mid \norm{u}_{\mathcal{F}H^{s}} \coloneqq \norm{\ev{x}^{s}u}_{L^2} < \infty \}$.
Let $C_{0}^{\infty}(\R^{d})$ be the space of smooth functions with compact support in $\R^{d}$.
$B(0,r)$ stands for the open ball in $\R^{d}$ with radius $r$ centered at the origin.
For $A \subset \R$, $\mathbbm{1}_{A}$ is the characteristic function of $A$.
Let $U(t)$ be the Schr\"odinger evolution operator $e^{it \Delta/2}$.
$A \lesssim B$ denotes $A \le CB$ for some constants $C>0$.

We also give the definition of solutions to \eqref{nls}.

\begin{definition}[Solution]\label{def:sol}
Let $I \ni 0$ be an interval.
Given $u_{0} \in L^{2}$,
we say a function $u \colon I \times \R^d \to \C$ is a solution to \eqref{nls} on $I$ if
$C(I; L^2) \cap L^{q_{0}}_{\mathrm{loc}}(I; L^{r_{0}})$
satisfies
\[
	u(t) = U(t) u_{0} - i \int_{0}^{t} U(t- s) F( u(s))\, ds
\]
in $L^2$ for any $t \in I$, where $(q_0, r_0) = ( 4(p+1)/d(p-1) , 1+p)$.
\end{definition}

If the nonlinearity $F$ satisfies \eqref{A1},
then \eqref{nls} is locally well-posed in $L^{2}$ (cf. \cite{YT87}).
Also, the solution $u$ belongs to $L^{q}_{\mathrm{loc}}(I; L^{r})$ for any admissible pair $(q,r)$, where a pair $(q,r)$ is said to be admissible if
\[
	2\le q,r \le \infty,\quad
    \frac{2}{q} = d\left(\frac12-\frac{1}{r} \right), \quad (d,q,r)\neq (2,2,\infty)
\]
(For details, see \cite{YT87}).
In particular, if $d = 1$, $2$ and $p>2$, then a pair $\left( \frac{4p}{d(p-2)}, p \right)$ is admissible.
Remark that $p \le 2$ always holds when $d \ge 3$, because of $p_{\mathrm{st}}(d) \le 2$.

We are in position to state the main result.

\begin{theorem} \label{thm:1}
Let $d \ge 1$.
Assume \eqref{A1} and \eqref{f:con} with $1 < p < p_{\mathrm{st}}(d)$.
Take $\alpha>0$ satisfying
\[
      \frac{2}{p-1} - \frac{d}{2} \le \alpha.
\]
Given $u_{0} \in \mathcal{F}H^{\alpha}$,
let $u \in C([0, \infty) ; L^{2})$ be a solution to \eqref{nls} with $u(0) = u_{0}$.
The following statements hold:
\begin{enumerate}[label=(\roman*)]
\item If $1< p \le 2$ and
\begin{align}
	\lim_{t \rightarrow \infty}	\norm{u - U(\cdot)u_{+}}_{L^{\infty}(t, \infty; L^{2})} = 0
	\label{asy:2}
\end{align}
for some $u_{+} \in L^{2}$,
then $u_{+} \equiv 0$.
\item If $2 < p < p_{\mathrm{st}}(d)$ (that is $d \le 2$) and
\begin{align}
	\lim_{t \rightarrow \infty}	\left( t^{\frac{d(p-2)}{4p}} \norm{u - U(\cdot)u_{+}}_{L^{\frac{4p}{d(p-2)}}(t, \infty; L^{p})} \right) = 0
	\label{asy:1}
\end{align}
for some $u_{+} \in \mathcal{F}H^{\beta}$ with $\beta > d/2-d/p$,
then $u_{+} \equiv 0$.
\end{enumerate}
\end{theorem}

\begin{remark}
When $p > 2$, the decay rate in \eqref{asy:1} coincides with that of $U(t)u_{+}$ in the same topology. Indeed, under the assumption $u_{+} \in \mathcal{F}H^{\beta} \setminus \{0\}$, the estimate
\[
	0 \neq \norm{\widehat{u_{+}}}_{L^{p}} \lesssim t^{\frac{4p}{d(p-2)}}\norm{U(\cdot)u_{+}}_{L^{\frac{4p}{d(p-2)}}(t, \infty; L^{p})} \lesssim \norm{\widehat{u_{+}}}_{L^{p}}
\]
holds for large $t$.
This implies that the condition \eqref{asy:1} is
the minimal assumption to ensure the uniqueness of the scattering state
$u_{+}$.
In fact, if $u_{+} \in \mathcal{F}H^{\beta}$ and we assume the slower decay condition
\[
      \lim_{t \rightarrow \infty} \left( t^{\gamma} \norm{u - U(\cdot)u_{+}}_{L^{\frac{4p}{d(p-2)}}(t, \infty; L^{p})} \right) = 0
\]
for some $\gamma < d(p-2)/(4p)$,
then the same condition holds when $u_{+}$ is replaced by any $\phi \in C_{0}^{\infty}(\R^{d})$.
Hence, the condition \eqref{asy:1} can be regarded as a natural assumption in the study of the scattering problem.
\end{remark}

\begin{remark}
When $1<p \le \min(2, 1+2/d)$, it was shown in \cite[Theorem 7.5.2]{C03} that
the gauge invariant case \eqref{ginls} with $\eta \in \C \setminus \{0\}$ admits no scattering states in $\Sigma \coloneqq H^{1} \cap \mathcal{F}H^{1}$ via the solution after the pseudo-conformal transformation.
The argument in \cite{C03} actually works for $u_0$, $u_+ \in L^2$, though for $d = 1$ and $2 < p \le 3$, the assumption $u_0$, $u_+ \in \mathcal{F}H^1$ appears to be necessary.
This range was already treated in the earlier work by Barab \cite{B84}, who proved non-scattering in the defocusing case $\eta>0$ using explicit time decay derived from the pseudo-conformal energy estimate under the weighted data.
Theorem \ref{thm:1} provides a complementary result for the non-gauge-invariant case.
\end{remark}

\begin{remark}
The weighted condition
\[
	0< s_c \coloneqq \frac{2}{p-1} - \frac{d}{2} \le \alpha
\]
is necessary for our argument. This range corresponds to the scale (super-)critical range in view of the weighted Sobolev space $\mathcal{F}H^{\alpha}$.
We emphasize that the initial data can be taken in the inhomogeneous scale-critical space $\mathcal{F}H^{s_c}$, particularly in $C_{0}^{\infty}(\R^{d})$,
which was always excluded in previous blow-up results for \eqref{nls}.
For instance, regarding \eqref{ngnls},
Ikeda and Inui \cite{II15a} proved the following:
If $f$ satisfies the condition
\begin{align}
	-\Im f(x) \ge
	\begin{cases}
	c |x|^{-k} & |x| > 1, \\
	0 & |x| \le 1
	\end{cases}
	\label{II:c}
\end{align}
for some $d/2 < k < 2/(p-1)$ and $c>0$, then
there exists a solution to \eqref{ngnls} with
$u(0) = f$
that blows up in finite time.
Hence, when taking $\alpha < s_{c}$,
by choosing $f(x) = i \mathbbm{1}_{\{|x| >1\}} |x|^{-k}$ with some $d/2 + \alpha < k < 2/(p-1)$, $f$ belongs to $\mathcal{F}H^{\alpha}$ and satisfies \eqref{II:c}. Thus, a blow-up solution with
$u(0) = f$
exists, but the $f$ does not belong to $\mathcal{F}H^{s_{c}}$.
Theorem \ref{thm:1} holds for any $u_{0} \in \mathcal{F}H^{\alpha}$ with $\alpha \ge s_{c}$,
and thus no restriction is imposed on the structure of $u_{0}$, unlike in \cite{II15a}.

\end{remark}

\begin{remark}
Our argument requires only $L^{2}$-solutions, although the initial data belongs to the weighted space $\mathcal{F}H^{\alpha}$.
In particular, we do not know whether local well-posedness for \eqref{nls} holds in $\mathcal{F}H^{\alpha}$, because the standard persistence of regularity argument does not apply to non-gauge-invariant nonlinearities.
\end{remark}

\begin{remark}
Without loss of generality, we may let $\theta =0$ in the assumption \eqref{f:con} by change of an unknown function $u \mapsto e^{-i \theta}u$.
When we consider the nonlinearity of \eqref{eq:decomp}, it also suffices to deal with the case $g_0=1$ via scaling.
\end{remark}

\begin{remark}
In \cite{KMM25}, Kawamoto, Masaki and the first author showed that
if $d \le 3$ and the nonlinearity is of the form \eqref{eq:decomp} with
\begin{align} \label{IC2}
p \in \left( \frac72,5 \right) \quad (d=1), \qquad
p \in (2,3) \quad (d=2), \qquad
p = 2 \quad (d=3),
\end{align}
then for any $u_{0} \in \Sigma \coloneqq H^{1} \cap \mathcal{F} H^{1}$,
\eqref{nls} has a global-in-time small solution in $\Sigma$ and scatters in the same topology,
provided that $g_{n} = 0$ for $n \leq 0$ and a certain summability condition on $\{g_{n}\}$ is imposed.
To the best of our knowledge,
the existence of global solutions in $\Sigma$ is not guaranteed if $g_{0} \neq 0$. Nevertheless, the existence of local-in-time solutions is shown by the standard well-posedness theory (cf. \cite{CW92, GOV94}).
Theorem \ref{thm:1} indicates that under $u_{0} \in \Sigma$, if $g_{0} \neq 0$, then any global solutions do not scatter in $\Sigma$, even if global solutions exist.
We impose the time decay condition \eqref{asy:1} in the sense of scattering if $d = 1$, $2$, but
Theorem \ref{thm:1} suggests that the result of \cite{KMM25} may not hold under $g_{0} \neq 0$.
\end{remark}

\section{Proof of the main theorem}

Before starting the discussion, we briefly give a crucial idea for treating the non-scattering phenomenon, along with an explanation of our perspective.
It is meaningful to clarify our viewpoint on the analysis of \eqref{NW}, as it helps to explain the main ideas of this paper.
In the proof of small data blow-up solutions to \eqref{NW} by Ikeda, the second author and Wakasa \cite[Proposition 2.1]{ISW2019}, one may observe that the Strauss exponent $p_{\rm st}(d-1)$ naturally appears when comparing the effect of nonlinearities with the behavior of solutions to the corresponding linear wave equation.
Specifically, the following estimates play an important role in the proof of \cite[Proposition 2.1]{ISW2019}:
\begin{align}
	\int_{\mathbb{R}^d} u_0\,dx + \int_{0}^T \eta_T^{2p'} \int_{\mathbb{R}^d} |u|^p\,dx\,dt &\le C T^{d-1-\frac{2}{p-1}},
	\label{est:w1} \\
	\int_{0}^T \eta_T^{2p'} \int_{\mathbb{R}^d} |u|^p\,dx\,dt &\ge \delta T^{d-\frac{d-1}{2}p},
	\label{est:w2}
\end{align}
where $\{\eta_T\}_{T>0}$ is a family of suitable cut-off functions in $t$.
The first estimate \eqref{est:w1} can be interpreted as a condition imposed by the nonlinearity, and a similar observation for \eqref{nls} has been addressed by Ikeda and Inui \cite{II15a, II15b}, and Ikeda and Wakasugi \cite{IW13}.
The second estimate \eqref{est:w2} reflects the large-time behavior of solutions to the corresponding linear wave equation.
By incorporating this perspective into the analysis of \eqref{nls}, we have realized that \eqref{est:w2} is closely related to solutions of \eqref{nls} that are asymptotically free. Thanks to this insight, we are able to prove the non-existence of scattering states for \eqref{nls} when $p < p_{\mathrm{st}}(d)$.
This perspective on \eqref{nls} is the novelty of the present paper.

\subsection{Preliminary}
Let us define a cut-off function $\eta$
satisfying
\[
      \eta \in C^{\infty}(\R), \quad
      \mathbbm{1}_{(- \infty, 1/2]} \le \eta(s) \le \mathbbm{1}_{(-\infty,1]}.
\]
Then we denote
\begin{align*}
	\psi_{R}(t,x) = \left\{ \eta \left( \frac{|x|^{2} + t}{R^{2}}\right) \right\}^{2p'}
\end{align*}
for any $R>0$, where $p'$ is the H\"older conjugate of $p$.
Such a cut-off function is firstly introduced in \cite{MP01} (cf. \cite{IS19}).
The function $\psi_{R}$ has the following property:

\begin{lemma} \label{lem:psi}
Let $R>0$ and $p>1$. Then
\begin{align*}
	\abs{\partial_{t} \psi_{R}(t,x)}
	&\lesssim \frac{1}{R^{2}} \left[ \psi_{R}(t,x) \right]^{1- \frac{1}{2p'}}, \qquad
	\abs{ \Delta \psi_{R}}
	\lesssim \frac{1}{R^{2}} \left[ \psi_{R}(t,x) \right]^{1- \frac{1}{p'}}
\end{align*}
hold for any $t \ge 0$ and $x \in \R^{d}$.
\end{lemma}
\begin{proof}
The desired assertion immediately follows from a direct calculation.
\end{proof}

To prove Theorem \ref{thm:1},
we shall introduce the definition of weak solutions to \eqref{nls}:

\begin{definition}[Weak solution]\label{def:wsol}
We say a function $u \in C([0, T) ; L^{2})$
is a weak solution to \eqref{nls}
on $[0, T)$, $T>0$,
if $u \in L^{q_{0}}_{\mathrm{loc}} ( (0, T) ; L^{r_{0}})$ and the identity
\begin{multline*}
	\iint_{(0, T) \times \R^d} u(t,x)(-i \partial_t \psi(t,x) + \Delta \psi(t,x))\, dxdt \\
	= i\int_{\R^d} u(0, x) \psi(0,x)\, dx + \iint_{(0, T) \times \R^d} F(u(t,x))\psi(t,x)\, dxdt
\end{multline*}
holds for any test function $\psi \in C_0^{\infty} ((-\infty, T) \times \R^{d})$, where $(q_0, r_0) = ( 4(p+1)/d(p-1) , 1+p)$.
\end{definition}

Remark that solutions in Definition \ref{def:sol}
always satisfy the property of
weak solutions in Definition \ref{def:wsol}
from the density argument involving Strichartz' estimates (see \cite{IW13}).
Hence, the solution $u$ satisfies
\begin{align}
	\begin{aligned}
	&\iint_{(0, R^{2}) \times \R^{d}} u(t,x) \left( -i \partial_t \psi_{R}(t,x) + \Delta \psi_{R}(t,x) \right)\, dxdt \\
	={}& i \int_{\R^d} u_{0}(x) \psi_{R}(0,x)\, dx
	+ \iint_{(0, R^{2}) \times \R^{d}} F(u(t,x))\psi_{R}(t,x)\, dxdt
	\end{aligned}
	\label{wsol:1}
\end{align}
for any $R>1$.
Set
\begin{align}
I(R) = \iint_{(0, R^{2}) \times \R^{d}} |u(t,x)|^{p} \psi_{R}(t,x)\, dxdt.
\label{def:IR}
\end{align}
The following lemma is essentially used in \cite{II15a, IW13} to argue by the test function method.
Here we give a sketch of the proof for the reader's convenience.
\begin{lemma} \label{lem:1}
The estimate
\begin{align*}
	\Re \iint_{(0, R^{2}) \times \R^{d}} u(t,x) \left( -i \partial_t \psi_{R}(t,x) + \Delta \psi_{R}(t,x) \right)\, dxdt
	\lesssim R^{d - (d+2) \frac{1}{p}} I(R)^{\frac{1}{p}}
\end{align*}
holds for any $R>1$.
\end{lemma}
\begin{proof}
Take $R>1$.
By Lemma \ref{lem:psi}, we see from H\"older's inequality that
\begin{align*}
	&\Re \iint_{(0, R^{2}) \times \R^{d}} u(t,x) \left( -i \partial_t \psi_{R}(t,x) + \Delta \psi_{R}(t,x) \right)\, dxdt \\
	\lesssim{}& \frac{1}{R^{2}} \iint_{(0, R^{2}) \times \R^{d}} |u(t,x)| \left[ \psi_{R}(t,x) \right]^{1-\frac{1}{2p'}}\, dxdt \\
	&\quad + \frac{1}{R^{2}} \iint_{(0, R^{2}) \times \R^{d}} |u(t,x)| \left[ \psi_{R}(t,x) \right]^{1-\frac{1}{p'}}\, dxdt \\
	&\lesssim \frac{1}{R^{2}}
	\left( \iint_{\left\{ (t,x) \mid |x|^{2} + t \le R^{2} \right\}} \, dxdt \right)^{\frac{1}{p'}}
      \left(  \iint_{\left\{ (t,x) \mid |x|^{2} + t \le R^{2} \right\}} |u(t,x)|^{p} \psi_{R}(t,x)\, dxdt \right)^{\frac{1}{p}} \\
	&\lesssim R^{d - (d+2) \frac{1}{p}} I(R)^{\frac{1}{p}},
\end{align*}
as desired.
\end{proof}

The slowly decaying condition given by \eqref{II:c} is assumed in \cite{II15a}, to treat the term involving initial data, which is the first term on the right-hand side of \eqref{wsol:1}.
This condition disturbs the establishment of global existence.
In contrast, we handle this term using a spatially weighted condition, which naturally arises in the context of scattering theory for nonlinear Schr\"odinger equations.
Under the weighted condition, the term turns out to be the most harmless in our argument.

\begin{lemma} \label{lem:2}
Let $\alpha >0$.
The estimate
\begin{align*}
\left| \Im \int_{\R^d}u_{0}(x)\psi_R(0,x)\, dx \right|
&\lesssim \norm{u_{0}}_{\mathcal{F}H^{\alpha}}
\begin{cases}
	1 & \text{if}\ \alpha > d/2, \\
      \log R & \text{if}\ \alpha = d/2, \\
	R^{\frac{d-2\alpha}{2}} & \text{if}\ \alpha < d/2
\end{cases}
\end{align*}
holds for any $R>2$.
\end{lemma}
\begin{proof}
By H\"older's inequality, we have
\begin{align*}
\left|\Im \int_{\R^d}u_{0}(x)\psi_R(0,x)\,dx\right|
&\le \int_{B(0,R)}|u_{0}(x)|\, dx \\
&\le \left(\int_{B(0,R)}|u_{0}(x)|^2\ev{x}^{2\alpha}\,dx\right)^{\frac{1}{2}}
\left(\int_{B(0,R)}\ev{x}^{-2\alpha}\,dx\right)^{\frac{1}{2}},
\end{align*}
from which the desired estimate follows.
\end{proof}

The following lemma plays a central role in the discussion of this paper, which describes a requirement for the non-zero scattering states.
\begin{lemma} \label{lem:3}
Let $u_{+} \in \mathcal{F}H^{\beta}$ with $\beta > d/2-d/p$ if $p>2$, otherwise $u_{+} \in L^{2}$.
Assume \eqref{asy:1} and \eqref{asy:2}.
If $u_+\not\equiv 0$, then the estimate
\begin{align*}
	I(R) \gtrsim R^{d+1-\frac{d}{2}p}
\end{align*}
holds for $R>0$ large enough.
\end{lemma}
\begin{proof}
Take $r_{0}>1$ such that $\norm{\widehat{u_{+}}}_{L^{p}(B(0,r_{0}))} \neq 0$.
Note that $\widehat{u_{+}} \in L^{p}(B(0,r_{0}))$ holds by the assumption.
Set
\[
	\mathcal{D}_{r_0,R}=
	\left\{  (t, x) \in (0, R^{2}) \times \R^{d} \relmiddle| |x|\le r_0t,\; \frac{R}{4r_0} \le t \le \frac{R}{2r_0} \right\}
\]
for $R> 2$.
Remark that $\psi_{R} \equiv 1$ on $\mathcal{D}_{r_0,R}$.
We define
\[
      \phi_{+}(t) = M(t)D(t)\mathcal{F}u_{+},
\]
which is the asymptotics of $U(t)u_{+}$,
where $M(t)$ and $D(t)$ are unitary operators on $L^2$ defined by
\[
	[M(t)f](x) = e^{i\frac{|x|^2}{2t}} f(x),\qquad
	[D(t)f](x) = (it)^{-\frac{d}2} f \left( \frac{x}{t} \right)
\]
for any $t \neq 0$, respectively.
Let us decompose $u$ as
\[
	u= \phi_{+}(t) + (U(t)u_{+} - \phi_{+}(t)) + (u - U(t)u_{+}).
\]
By the triangle inequality in $L^p(\mathcal{D}_{r_0,R})$, we have
\begin{align}
\begin{aligned}
	I(R)^{\frac{1}{p}}
	\ge{}&
	\left(\iint_{\mathcal{D}_{r_0,R}}|\phi_{+}(t)|^p\,dxdt\right)^{\frac{1}{p}}
	\\
	&-
	\left(\iint_{\mathcal{D}_{r_0,R}} |U(t)u_{+} - \phi_{+}(t)|^{p}\,dxdt\right)^{\frac{1}{p}}
	- \left(\iint_{\mathcal{D}_{r_0,R}} |u - U(t)u_{+}|^{p}\, dxdt\right)^{\frac{1}{p}}.
\end{aligned}
	\label{remain:1}
\end{align}
Let us first estimate the last term in the right-hand side of \eqref{remain:1}.
We shall treat the case $p \le 2$.
Take $\varepsilon>0$ small enough.
By the assumption \eqref{asy:2}, there exists $t_{\ast}(\varepsilon)>0$ such that if $t \ge t_{\ast}(\varepsilon)$, then
\begin{align}
	\norm{u(t) - U(t)u_{+}}_{L^{2}} < \varepsilon.
\end{align}
Hence, H\"older's inequality leads to
\begin{align}
	\begin{aligned}
	\iint_{\mathcal{D}_{r_0,R}} |u(t)-U(t)u_+|^p\, dxdt
	&= \int_{\frac{R}{4r_{0}}}^{\frac{R}{2r_{0}}} \left( \int_{B(0,r_0t)} |u(t)-U(t)u_+|^p\,dx\right)\, dt \\
	&\lesssim \int_{\frac{R}{4r_{0}}}^{\frac{R}{2r_{0}}} \left( \int_{B(0,r_0t)}|u(t)-U(t)u_+|^2\, dx \right)^{\frac{p}{2}}
	(r_0t)^{d \left(  1-\frac{p}{2}\right)}\,dt \\
	&\lesssim \sup_{t \ge \frac{R}{4r_{0}}} \norm{u(t)-U(t) u_+}_{L^{2}}^p R^{d+1- \frac{d}{2}p} \\
	&< \varepsilon R^{d+1- \frac{d}{2}p}
	\end{aligned}
	\label{remain:2}
\end{align}
for any $R \ge 4r_{0} t_{\ast}(\varepsilon)$.
A direct calculation shows that $U(t) = M(t)D(t) \mathcal{F}M(t)$, and by Plancharel's theorem,
the second term in \eqref{remain:1} can be estimated as follows:
\begin{align*}
\iint_{\mathcal{D}_{r_0,R}}|U(t)u_{+} - \phi_{+}(t)|^p\, dxdt
	&= \iint_{\mathcal{D}_{r_0,R}}|D(t)\mathcal{F}(M(t)u_+-u_+)|^p\, dxdt \\
	&= \int_{\frac{R}{4r_0}}^{\frac{R}{2r_0}}
	t^{d-\frac{d}{2}p} \int_{B(0,r_0)}|\mathcal{F}(M(t)u_+-u_+)|^p\, dxdt \\
	&\lesssim \int_{\frac{R}{4r_0}}^{\frac{R}{2r_0}} (r_{0}t)^{d-\frac{d}{2}p}
	\norm{M(t)u_+-u_+}_{L^2}^p\, dt \\
	&\lesssim R^{d+1-\frac{d}{2}p}
	\left(\sup_{t \ge R/4r_0}\|M(t)u_+-u_+\|_{L^2}\right)^p.
\end{align*}
Since $u_{+} \in L^{2}$, $M(t) \rightarrow 1$ as $t \rightarrow \infty$ and $|M(t)| =1$, Lebesgue's dominated convergence theorem implies that for any $\varepsilon>0$, there exists $R_{\ast}(\varepsilon) > 1$ such that if $R \ge R_{\ast}(\varepsilon)$, then
\[
	\sup_{t \ge R/4r_0}\|M(t)u_+-u_+\|_{L^2} < \varepsilon.
\]
Then we conclude that
\begin{align}
	\iint_{\mathcal{D}_{r_0,R}}|U(t)u_+-\phi_+|^p\, dxdt
	\lesssim \varepsilon R^{d+1-\frac{d}{2}p}
	\label{remain:3}
\end{align}
for $R>0$ large enough.

Let us deal with the case $p > 2$, which occurs only in $d=1$, $2$. Set $q = \frac{4p}{d(p-2)}$.
Note that
\[
	\frac{2}{q}+\frac{d}{p} = \frac{d}{2}
	\iff -\frac{2p}{q} = d -\frac{d}{2}p.
\]
By \eqref{asy:1}, we have
\begin{align*}
	\iint_{\mathcal{D}_{r_0,R}}|u(t)-U(t)u_+|^p\,dxdt
	&\le \int_{\frac{R}{4r_{0}}}^{\frac{R}{2r_{0}}} \norm{u(t)-U(t)u_+}_{L^p}^p\,dt \\
	&\le \left( \int_{\frac{R}{4r_{0}}}^{\frac{R}{2r_{0}}}
	\norm{u(t)-U(t)u_+}_{L^p}^q\,dt \right)^{\frac{p}{q}}
	\left( \int_{\frac{R}{4r_{0}}}^{\frac{R}{2r_{0}}}\,dt \right)^{1-\frac{p}{q}} \\
	&\lesssim \left( R^{\frac{1}{q}}\|u(t)-U(t)u_+\|_{L^q(R/4r_0,\infty;L^p)} \right)^p
R^{1-\frac{2p}{q}} \\
	&\lesssim \varepsilon R^{d+1-\frac{d}{2}p}
\end{align*}
for any $R \ge 4r_{0} t_{\ast}(\varepsilon)$.
Also, one sees from Young's inequality that
\begin{align*}
\iint_{\mathcal{D}_{r_0,R}}|U(t)u_{+} - \phi_{+}(t)|^p\, dxdt
	&= \iint_{\mathcal{D}_{r_0,R}}|D(t)\mathcal{F}(M(t)u_+-u_+)|^p\, dxdt \\
	&= \int_{\frac{R}{4r_0}}^{\frac{R}{2r_0}}
	t^{d-\frac{d}{2}p} \int_{B(0,r_0)}|\mathcal{F}(M(t)u_+-u_+)|^p\, dxdt \\
	&\lesssim R^{d+1-\frac{d}{2}p}
	\left(\sup_{t\ge R/4r_0}\|M(t)u_+-u_+\|_{L^{p'}}\right)^p.
\end{align*}

Since $u_{+} \in \mathcal{F}H^{\beta} \hookrightarrow L^{p'}$, similarly to the above,
for any $\varepsilon>0$, there exists $R_{\ast}(\varepsilon) > 1$ such that if $R \ge R_{\ast}(\varepsilon)$, then
\[
	\sup_{t \ge R/4r_0}\|M(t)u_+-u_+\|_{L^{p'}} < \varepsilon,
\]
which implies \eqref{remain:3}.

Finally, the first term of \eqref{remain:1} can be explicitly calculated as
\begin{align}
	\begin{aligned}
	\iint_{\mathcal{D}_{r_0,R}}|\phi_{+}(t)|^p\,dxdt &= \iint_{\mathcal{D}_{r_0,R}} |D(t)\mathcal{F}u_+|^{p} \,dxdt \\
	&= \int_{R/4r_0}^{R/2r_0} t^{d-\frac{d}{2}p}\int_{\{|x|<r_0\}} |\mathcal{F}u_+|^{p}\, dxdt \\
	&= C R^{d+1-\frac{d}{2}p} \norm{\widehat{u_{+}}}_{L^p(B(0,r_0))}^{p}.
	\end{aligned}
	\label{remain:4}
\end{align}
Substituting \eqref{remain:2}, \eqref{remain:3} and \eqref{remain:4} into \eqref{remain:1},
we establish the desired estimate.
This completes the proof.
\end{proof}

\subsection{Proof of Theorem \ref{thm:1}}
Before starting the proof, we highlight the contrast between the gauge-invariant case \eqref{ginls} and the non-oscillatory case \eqref{nls}.
When the nonlinearity lacks oscillatory components, it becomes difficult to capture the interaction between solutions through the nonlinearity.
However, this feature allows us to analyze the size of the nonlinearity directly, and as shown in \eqref{remain:4},
we can extract a lower bound for the space-time integral of the nonlinearity from the long-time behavior of the solution.
This stands in sharp contrast to the gauge-invariant case and is precisely what enables the following argument to work.

\begin{proof}[Proof of Theorem \ref{thm:1}]
Assume $u_{+} \not\equiv 0$ and let us argue by contradiction.
Taking the real part of the both side of \eqref{wsol:1}, we see from Lemma \ref{lem:1} and \eqref{f:con} that
\begin{align*}
	I(R) &\lesssim \Re \iint_{(0, R^{2}) \times \R^{d}} F(u(t,x)) \psi_{R}(t,x)\, dxdt \\
	&\lesssim \left| \Im \int_{\R^d}u_{0}(x)\psi_R(0,x)\, dx \right| + R^{d - (d+2) \frac{1}{p}} I(R)^{\frac{1}{p}}
\end{align*}
for any $R>1$.
By Young's inequality, one has
\begin{align*}
	R^{d - (d+2) \frac{1}{p}} I(R)^{\frac{1}{p}}
	\le \varepsilon I(R) + C_{\varepsilon} R^{\frac{p}{p-1} \left( d-(d+2)\frac{1}{p} \right)}
	= \varepsilon I(R) + C_{\varepsilon} R^{d-\frac{2}{p-1}}
\end{align*}
for small $\varepsilon>0$.
These imply
\begin{align*}
	I(R) \lesssim \left| \Im \int_{\R^d}u_{0}(x)\psi_R(0,x)\, dx \right| + R^{d-\frac{2}{p-1}}.
\end{align*}

Let us only consider the case $p>1 + 2/d$ and $\alpha < d/2$, because the other cases are easier.
By Lemma \ref{lem:2}, we have
\begin{align*}
	\iint_{\mathcal{D}_{r_0,R}}|u|^p\,dxdt
	\lesssim R^{d-\frac{2}{p-1}}+R^{\frac{d-2\alpha}{2}}
\end{align*}
for any $R >2$.
Note that
\[
	d-\frac{2}{p-1} \ge \frac{d-2\alpha}{2}
\]
if $p-1 \ge 4/(d+2 \alpha) \iff \alpha \ge 2/(p-1) - d/2$.
Then Lemma \ref{lem:3} yields
\[
	R^{d+1 - \frac{d}{2}p} \lesssim R^{d-\frac{2}{p-1}}+R^{\frac{d-2\alpha}{2}} \lesssim R^{d-\frac{2}{p-1}}
\]
for large $R$.
This leads to
\[
	d+1 - \frac{d}{2}p \le d-\frac{2}{p-1}
      \iff p \ge p_{\mathrm{st}}(d),
\]
which contradicts the assumption $p < p_{\mathrm{st}}(d)$.
Thus $u_{+} \equiv 0$.
The proof is completed.
\end{proof}

\section*{Acknowledgments}
H.M. was supported by JSPS KAKENHI Grant Number 22K13941.
The authors would like to express their sincere appreciation to the anonymous referee for a careful reading of the manuscript and for providing thoughtful suggestions.
They are also sincerely grateful to Professor Takahisa Inui for his insightful comments regarding the assumption on the nonlinearity, and to Professor Satoshi Masaki for his valuable feedback on an earlier version of the manuscript.


\begin{bibdiv}
\begin{biblist}

\bib{B84}{article}{
      author={Barab, Jacqueline~E.},
       title={Nonexistence of asymptotically free solutions for a nonlinear {S}chr\"{o}dinger equation},
        date={1984},
        ISSN={0022-2488},
     journal={J. Math. Phys.},
      volume={25},
      number={11},
       pages={3270\ndash 3273},
         url={https://doi.org/10.1063/1.526074},
      review={\MR{761850}},
}

\bib{BGTV23}{article}{
      author={Burq, N.},
      author={Georgiev, V.},
      author={Tzvetkov, N.},
      author={Visciglia, N.},
       title={{$H^1$} scattering for mass-subcritical {NLS} with short-range nonlinearity and initial data in {$\Sigma$}},
        date={2023},
        ISSN={1424-0637},
     journal={Ann. Henri Poincar\'{e}},
      volume={24},
      number={4},
       pages={1355\ndash 1376},
         url={https://doi.org/10.1007/s00023-022-01245-2},
      review={\MR{4575488}},
}

\bib{C03}{book}{
      author={Cazenave, Thierry},
       title={Semilinear {S}chr\"{o}dinger equations},
      series={Courant Lecture Notes in Mathematics},
   publisher={New York University, Courant Institute of Mathematical Sciences, New York; American Mathematical Society, Providence, RI},
        date={2003},
      volume={10},
        ISBN={0-8218-3399-5},
         url={https://doi.org/10.1090/cln/010},
      review={\MR{2002047}},
}

\bib{CW92}{article}{
      author={Cazenave, Thierry},
      author={Weissler, Fred~B.},
       title={Rapidly decaying solutions of the nonlinear {S}chr\"{o}dinger equation},
        date={1992},
        ISSN={0010-3616},
     journal={Comm. Math. Phys.},
      volume={147},
      number={1},
       pages={75\ndash 100},
         url={http://projecteuclid.org/euclid.cmp/1104250527},
      review={\MR{1171761}},
}

\bib{DL81}{article}{
      author={Dong, Guang~Chang},
      author={Li, Shu~Jie},
       title={On the initial value problem for a nonlinear {S}chr\"{o}dinger equation},
        date={1981},
        ISSN={0022-0396},
     journal={J. Differential Equations},
      volume={42},
      number={3},
       pages={353\ndash 365},
         url={https://doi.org/10.1016/0022-0396(81)90109-1},
      review={\MR{639226}},
}

\bib{FO16}{article}{
      author={Fujiwara, Kazumasa},
      author={Ozawa, Tohru},
       title={Finite time blowup of solutions to the nonlinear {S}chr\"{o}dinger equation without gauge invariance},
        date={2016},
        ISSN={0022-2488},
     journal={J. Math. Phys.},
      volume={57},
      number={8},
       pages={082103, 8},
         url={https://doi.org/10.1063/1.4960725},
      review={\MR{3535686}},
}

\bib{GLS97}{article}{
      author={Georgiev, Vladimir},
      author={Lindblad, Hans},
      author={Sogge, Christopher~D.},
       title={Weighted {S}trichartz estimates and global existence for semilinear wave equations},
        date={1997},
        ISSN={0002-9327},
     journal={Amer. J. Math.},
      volume={119},
      number={6},
       pages={1291\ndash 1319},
         url={http://muse.jhu.edu/journals/american_journal_of_mathematics/v119/119.6georgiev.pdf},
      review={\MR{1481816}},
}

\bib{GOV94}{article}{
      author={Ginibre, J.},
      author={Ozawa, T.},
      author={Velo, G.},
       title={On the existence of the wave operators for a class of nonlinear {S}chr\"{o}dinger equations},
        date={1994},
        ISSN={0246-0211},
     journal={Ann. Inst. H. Poincar\'{e} Phys. Th\'{e}or.},
      volume={60},
      number={2},
       pages={211\ndash 239},
         url={http://www.numdam.org/item?id=AIHPA_1994__60_2_211_0},
      review={\MR{1270296}},
}

\bib{Glassey81b}{article}{
      author={Glassey, Robert~T.},
       title={Existence in the large for {$ cmu=F(u)$} in two space dimensions},
        date={1981},
        ISSN={0025-5874},
     journal={Math. Z.},
      volume={178},
      number={2},
       pages={233\ndash 261},
         url={https://doi.org/10.1007/BF01262042},
      review={\MR{631631}},
}

\bib{Glassey81a}{article}{
      author={Glassey, Robert~T.},
       title={Finite-time blow-up for solutions of nonlinear wave equations},
        date={1981},
        ISSN={0025-5874},
     journal={Math. Z.},
      volume={177},
      number={3},
       pages={323\ndash 340},
         url={https://doi.org/10.1007/BF01162066},
      review={\MR{618199}},
}

\bib{HT87}{incollection}{
      author={Hayashi, Nakao},
      author={Tsutsumi, Yoshio},
       title={Remarks on the scattering problem for nonlinear {S}chr\"{o}dinger equations},
        date={1987},
   booktitle={Differential equations and mathematical physics ({B}irmingham, {A}la., 1986)},
      series={Lecture Notes in Math.},
      volume={1285},
   publisher={Springer, Berlin},
       pages={162\ndash 168},
         url={https://doi.org/10.1007/BFb0080593},
      review={\MR{921265}},
}

\bib{II15a}{article}{
      author={Ikeda, Masahiro},
      author={Inui, Takahisa},
       title={Small data blow-up of {$L^2$} or {$H^1$}-solution for the semilinear {S}chr\"{o}dinger equation without gauge invariance},
        date={2015},
        ISSN={1424-3199},
     journal={J. Evol. Equ.},
      volume={15},
      number={3},
       pages={571\ndash 581},
      review={\MR{3394699}},
}

\bib{II15b}{article}{
      author={Ikeda, Masahiro},
      author={Inui, Takahisa},
       title={Some non-existence results for the semilinear {S}chr\"{o}dinger equation without gauge invariance},
        date={2015},
        ISSN={0022-247X},
     journal={J. Math. Anal. Appl.},
      volume={425},
      number={2},
       pages={758\ndash 773},
         url={https://doi.org/10.1016/j.jmaa.2015.01.003},
      review={\MR{3303890}},
}

\bib{IS19}{article}{
      author={Ikeda, Masahiro},
      author={Sobajima, Motohiro},
       title={Sharp upper bound for lifespan of solutions to some critical semilinear parabolic, dispersive and hyperbolic equations via a test function method},
        date={2019},
        ISSN={0362-546X},
     journal={Nonlinear Anal.},
      volume={182},
       pages={57\ndash 74},
         url={https://doi.org/10.1016/j.na.2018.12.009},
      review={\MR{3894246}},
}

\bib{ISW2019}{article}{
      author={Ikeda, Masahiro},
      author={Sobajima, Motohiro},
      author={Wakasa, Kyouhei},
       title={Blow-up phenomena of semilinear wave equations and their weakly coupled systems},
        date={2019},
        ISSN={0022-0396},
     journal={J. Differential Equations},
      volume={267},
      number={9},
       pages={5165\ndash 5201},
         url={https://doi.org/10.1016/j.jde.2019.05.029},
      review={\MR{3991556}},
}

\bib{IW13}{article}{
      author={Ikeda, Masahiro},
      author={Wakasugi, Yuta},
       title={Small-data blow-up of {$L^2$}-solution for the nonlinear {S}chr\"{o}dinger equation without gauge invariance},
        date={2013},
        ISSN={0893-4983},
     journal={Differential Integral Equations},
      volume={26},
      number={11-12},
       pages={1275\ndash 1285},
         url={http://projecteuclid.org/euclid.die/1378327426},
      review={\MR{3129009}},
}

\bib{John1979}{article}{
      author={John, Fritz},
       title={Blow-up of solutions of nonlinear wave equations in three space dimensions},
        date={1979},
        ISSN={0025-2611},
     journal={Manuscripta Math.},
      volume={28},
      number={1-3},
       pages={235\ndash 268},
         url={https://doi.org/10.1007/BF01647974},
      review={\MR{535704}},
}

\bib{Kato80}{article}{
      author={Kato, Tosio},
       title={Blow-up of solutions of some nonlinear hyperbolic equations},
        date={1980},
        ISSN={0010-3640},
     journal={Comm. Pure Appl. Math.},
      volume={33},
      number={4},
       pages={501\ndash 505},
         url={https://doi.org/10.1002/cpa.3160330403},
      review={\MR{575735}},
}

\bib{Ka94}{incollection}{
      author={Kato, Tosio},
       title={An {$L^{q,r}$}-theory for nonlinear {S}chr\"{o}dinger equations},
        date={1994},
   booktitle={Spectral and scattering theory and applications},
      series={Adv. Stud. Pure Math.},
      volume={23},
   publisher={Math. Soc. Japan, Tokyo},
       pages={223\ndash 238},
         url={https://doi.org/10.2969/aspm/02310223},
      review={\MR{1275405}},
}

\bib{KMM25}{article}{
      author={Kawamoto, Masaki},
      author={Masaki, Satoshi},
      author={Miyazaki, Hayato},
       title={Global well-posedness and scattering in weighted space for nonlinear {S}chr\"{o}dinger equations below the {S}trauss exponent without gauge-invariance},
        date={2025},
        ISSN={0025-5831},
     journal={Math. Ann.},
      volume={392},
      number={1},
       pages={1051\ndash 1097},
         url={https://doi.org/10.1007/s00208-025-03121-w},
      review={\MR{4887783}},
}

\bib{MM}{article}{
      author={Masaki, Satoshi},
      author={Miyazaki, Hayato},
       title={Long range scattering for nonlinear {S}chr\"{o}dinger equations with critical homogeneous nonlinearity},
        date={2018},
        ISSN={0036-1410},
     journal={SIAM J. Math. Anal.},
      volume={50},
      number={3},
       pages={3251\ndash 3270},
         url={https://doi.org/10.1137/17M1144829},
      review={\MR{3815545}},
}

\bib{MM2}{article}{
      author={Masaki, Satoshi},
      author={Miyazaki, Hayato},
       title={Nonexistence of scattering and modified scattering states for some nonlinear {S}chr\"{o}dinger equation with critical homogeneous nonlinearity},
        date={2019},
        ISSN={0893-4983},
     journal={Differential Integral Equations},
      volume={32},
      number={3-4},
       pages={121\ndash 138},
         url={https://projecteuclid.org/euclid.die/1548212426},
      review={\MR{3909981}},
}

\bib{MMU}{article}{
      author={Masaki, Satoshi},
      author={Miyazaki, Hayato},
      author={Uriya, Kota},
       title={Long-range scattering for nonlinear {S}chr\"{o}dinger equations with critical homogeneous nonlinearity in three space dimensions},
        date={2019},
        ISSN={0002-9947},
     journal={Trans. Amer. Math. Soc.},
      volume={371},
      number={11},
       pages={7925\ndash 7947},
         url={https://doi.org/10.1090/tran/7636},
      review={\MR{3955539}},
}

\bib{MP01}{article}{
      author={Mitidieri, \`E.},
      author={Pokhozhaev, S.~I.},
       title={A priori estimates and the absence of solutions of nonlinear partial differential equations and inequalities},
        date={2001},
        ISSN={0371-9685},
     journal={Tr. Mat. Inst. Steklova},
      volume={234},
       pages={1\ndash 384},
      review={\MR{1879326}},
}

\bib{MN21}{article}{
      author={Murphy, Jason},
      author={Nakanishi, Kenji},
       title={Failure of scattering to solitary waves for long-range nonlinear {S}chr\"{o}dinger equations},
        date={2021},
        ISSN={1078-0947},
     journal={Discrete Contin. Dyn. Syst.},
      volume={41},
      number={3},
       pages={1507\ndash 1517},
         url={https://doi.org/10.3934/dcds.2020328},
      review={\MR{4201850}},
}

\bib{NO02}{article}{
      author={Nakanishi, Kenji},
      author={Ozawa, Tohru},
       title={Remarks on scattering for nonlinear {S}chr\"{o}dinger equations},
        date={2002},
        ISSN={1021-9722},
     journal={NoDEA Nonlinear Differential Equations Appl.},
      volume={9},
      number={1},
       pages={45\ndash 68},
         url={https://doi.org/10.1007/s00030-002-8118-9},
      review={\MR{1891695}},
}

\bib{Rammaha89}{incollection}{
      author={Rammaha, M.~A.},
       title={Nonlinear wave equations in high dimensions},
        date={1989},
   booktitle={Differential equations and applications, {V}ol. {I}, {II} ({C}olumbus, {OH}, 1988)},
   publisher={Ohio Univ. Press, Athens, OH},
       pages={322\ndash 326},
      review={\MR{1026235}},
}

\bib{Schaeffer85}{article}{
      author={Schaeffer, Jack},
       title={The equation {$u_{tt}-\Delta u=|u|^p$} for the critical value of {$p$}},
        date={1985},
        ISSN={0308-2105},
     journal={Proc. Roy. Soc. Edinburgh Sect. A},
      volume={101},
      number={1-2},
       pages={31\ndash 44},
         url={https://doi.org/10.1017/S0308210500026135},
      review={\MR{824205}},
}

\bib{S05}{article}{
      author={Shimomura, Akihiro},
       title={Nonexistence of asymptotically free solutions for quadratic nonlinear {S}chr\"{o}dinger equations in two space dimensions},
        date={2005},
        ISSN={0893-4983},
     journal={Differential Integral Equations},
      volume={18},
      number={3},
       pages={325\ndash 335},
      review={\MR{2122723}},
}

\bib{ST06}{article}{
      author={Shimomura, Akihiro},
      author={Tsutsumi, Yoshio},
       title={Nonexistence of scattering states for some quadratic nonlinear {S}chr\"{o}dinger equations in two space dimensions},
        date={2006},
        ISSN={0893-4983},
     journal={Differential Integral Equations},
      volume={19},
      number={9},
       pages={1047\ndash 1060},
      review={\MR{2262096}},
}

\bib{Sideris84}{article}{
      author={Sideris, Thomas~C.},
       title={Nonexistence of global solutions to semilinear wave equations in high dimensions},
        date={1984},
        ISSN={0022-0396},
     journal={J. Differential Equations},
      volume={52},
      number={3},
       pages={378\ndash 406},
         url={https://doi.org/10.1016/0022-0396(84)90169-4},
      review={\MR{744303}},
}

\bib{S74}{inproceedings}{
      author={Strauss, W.~A.},
       title={Nonlinear scattering theory},
        date={1974},
   booktitle={Scattering theory in mathematical physics},
      editor={Lavita, J.~A.},
      editor={Marchand, J.-P.},
   publisher={Springer Netherlands},
     address={Dordrecht},
       pages={53\ndash 78},
}

\bib{S81a}{article}{
      author={Strauss, Walter~A.},
       title={Nonlinear scattering theory at low energy},
        date={1981},
        ISSN={0022-1236},
     journal={J. Functional Analysis},
      volume={41},
      number={1},
       pages={110\ndash 133},
         url={https://doi.org/10.1016/0022-1236(81)90063-X},
      review={\MR{614228}},
}

\bib{YT85}{article}{
      author={Tsutsumi, Yoshio},
       title={Scattering problem for nonlinear {S}chr\"{o}dinger equations},
        date={1985},
        ISSN={0246-0211},
     journal={Ann. Inst. H. Poincar\'{e} Phys. Th\'{e}or.},
      volume={43},
      number={3},
       pages={321\ndash 347},
         url={http://www.numdam.org/item?id=AIHPB_1985__43_3_321_0},
      review={\MR{824843}},
}

\bib{YT87}{article}{
      author={Tsutsumi, Yoshio},
       title={{$L^2$}-solutions for nonlinear {S}chr\"{o}dinger equations and nonlinear groups},
        date={1987},
        ISSN={0532-8721},
     journal={Funkcial. Ekvac.},
      volume={30},
      number={1},
       pages={115\ndash 125},
         url={http://www.math.kobe-u.ac.jp/~fe/xml/mr0915266.xml},
      review={\MR{915266}},
}

\bib{TY84}{article}{
      author={Tsutsumi, Yoshio},
      author={Yajima, Kenji},
       title={The asymptotic behavior of nonlinear {S}chr\"odinger equations},
        date={1984},
        ISSN={0273-0979},
     journal={Bull. Amer. Math. Soc. (N.S.)},
      volume={11},
      number={1},
       pages={186\ndash 188},
         url={http://dx.doi.org/10.1090/S0273-0979-1984-15263-7},
      review={\MR{741737}},
}

\bib{YZ06}{article}{
      author={Yordanov, Borislav~T.},
      author={Zhang, Qi~S.},
       title={Finite time blow up for critical wave equations in high dimensions},
        date={2006},
        ISSN={0022-1236},
     journal={J. Funct. Anal.},
      volume={231},
      number={2},
       pages={361\ndash 374},
         url={https://doi.org/10.1016/j.jfa.2005.03.012},
      review={\MR{2195336}},
}

\bib{Zhou2007}{article}{
      author={Zhou, Yi},
       title={Blow up of solutions to semilinear wave equations with critical exponent in high dimensions},
        date={2007},
        ISSN={0252-9599},
     journal={Chinese Ann. Math. Ser. B},
      volume={28},
      number={2},
       pages={205\ndash 212},
         url={https://doi.org/10.1007/s11401-005-0205-x},
      review={\MR{2316656}},
}

\end{biblist}
\end{bibdiv}


\end{document}